\documentclass[11pt,english]{article}
\usepackage{mathptmx}

\usepackage[T1]{fontenc}
\usepackage[latin9]{inputenc}
\usepackage[a4paper]{geometry}
\usepackage{amsmath}
\usepackage{amsthm}
\usepackage{amssymb}
\usepackage{stackrel}

\makeatletter
\numberwithin{equation}{section}
\numberwithin{figure}{section}
\theoremstyle{plain}
\newtheorem{thm}{\protect\theoremname}
\theoremstyle{plain}
\newtheorem{lem}[thm]{\protect\lemmaname}
\theoremstyle{plain}
\newtheorem{cor}[thm]{\protect\corollaryname}
\theoremstyle{plain}
\newtheorem{prop}[thm]{\protect\propositionname}
\theoremstyle{remark}
\newtheorem{rem}[thm]{\protect\remarkname}

\makeatother

\usepackage{babel}
\providecommand{\corollaryname}{Corollary}
\providecommand{\lemmaname}{Lemma}
\providecommand{\propositionname}{Proposition}
\providecommand{\remarkname}{Remark}
\providecommand{\theoremname}{Theorem}

\begin{document}
\title{A class of multidimensional quadratic BSDEs}
\author{Zhongmin Qian \thanks{Exeter College, Turl Street, Oxford OX1 3DP. Research supported partially
by ERC grant ESig ID 291244. Email: $\mathtt{qianz@maths.ox.ac.uk}$}, Shujin Wu\thanks{$ $School of Finance and Statistics, East China Normal University,
500 Dongchuan Road, Shanghai 200241. Email: $\mathtt{sjwu@stat.ecnu.edu.cn}$} $\mbox{ \ensuremath{}\ensuremath{}}$and Yimin Yang\thanks{$ $Mathematical Institute, University of Oxford, Oxford OX2 6GG.
Email: $\mathtt{yimin.yang@maths.ox.ac.uk}$} $\mbox{ \ensuremath{}\ensuremath{}}$}
\maketitle
\begin{abstract}
In this paper we study a multidimensional quadratic BSDE with a particular
class of product generators and give a result of existence of solution
in a suitable complete metric space under some constraints on parameters.
We also use that result to derive the existence and uniqueness of
solution to the one dimensional case with bounded terminal values
and show the existence of solution to a lower triangular quadratic
BSDE with certain bounded terminal values.
\end{abstract}

\section{Introduction}

A multidimensional quadratic Backward Stochastic Differential Equation
(BSDE) on $\left[0,T\right]$ with $T$ being the terminal time, according
to the formulation put forward by Pardoux and Peng \cite{key-11},
is a stochastic integral equation with
\begin{equation}
Y_{t}=\xi+\int_{t}^{T}f\left(s,Y_{s},Z_{s}\right)ds-\int_{t}^{T}Z_{s}dW_{s}\label{eq:1}
\end{equation}
where $Y_{t}$ is $\mbox{\ensuremath{\boldsymbol{R}}}^{d}$-valued,
$Z_{t}$ is $\mbox{\ensuremath{\boldsymbol{R}}}^{d\times k}$-valued,
terminal value $\xi$ is $\mbox{\ensuremath{\boldsymbol{R}}}^{d}$-valued
and $\mathcal{F}_{T}$ measurable. The generator function $f:\left[0,T\right]\times\Omega\times\ensuremath{\boldsymbol{R}^{d}}\times\mbox{\ensuremath{\boldsymbol{R}}}^{d\times k}\rightarrow\ensuremath{\boldsymbol{R}}^{d}$
is of quadratic growth and $W$ is a standard $k$-dimensional Brownian
motion defined on $\left(\Omega,\mathcal{F},\left(\mathcal{F}_{t}\right),\mathbb{P}\right)$
where $\left(\mathcal{F}_{t}\right)$ is the Brownian filtration.

BSDE with quadratic growth can be used to solve problems such as utility
maximization with exponential utility function. They were studied
by Kobylanski \cite{key-10} and extended by others for example \cite{key-3}\cite{key-4}\cite{key-7}
and etc. More precisely in 2000, by using the monotonicity method
adopted from PDE theory, Kobylanski \cite{key-10} solved a class
of one dimensional BSDEs with generator function being of quadratic
growth in $Z$. This particular class of quadratic BSDEs with unbound
terminal values were further studied by Briand and Hu \cite{key-3}\cite{key-4}
and Delbaen, Hu and Richou \cite{key-7}. In 2013, Barrieu and El
Karoui \cite{key-1} adopted a different approach to prove the existence
under conditions similar to those of Briand and Hu \cite{key-3},
while Briand and Elie \cite{key-2} gave a concise study for the case
when the terminal value $\xi$ is bounded. The method used in the
present paper to get the main result was partially inspired by the
method in Tevzadze \cite{key-12} for solving existence of solutions
to a quadratic BSDE driven by a continuous martingale with bounded
terminal values. The case of multidimensional quadratic BSDEs seems
significantly more difficult than that of Lipschitz BSDEs, and the
methods used in literature are often quite involved, and up until
now the results about quadratic BSDEs are mostly only for the one-dimensional
case, and they heavily relay on comparison theorems. In 2015, P. Cheridito
and K. Nam \cite{key-5} discussed special systems of BSDEs assuming
Markovian and subquadraticity because of filtration issue. Recently,
based on a result for one-dimensional BSDEs in Briand and Hu \cite{key-3},
Hu and Tang \cite{key-8} proved the existence and uniqueness of solution
to a multidimensional BSDE with diagonal quadratic generator assuming
that each component $f^{i}$ of the generator $f$ depends only on
the $i$th row of the matrix variable $Z$ in the BSDE (\ref{eq:1}).

In this paper we study a multidimensional quadratic BSDE with a particular
class of product generators and give a result of existence of solution
in a suitable complete metric space under some constraints on parameters.
The corresponding PDEs however have significance in fluid dynamics
and in fact they are simplified version of fluid equations. We also
use that result to derive the existence and uniqueness of solution
to the one dimensional case of our BSDE with bounded terminal values
and then use the result for the one dimensional case to show the existence
of solution to a lower triangular quadratic BSDE with certain bounded
terminal values. The paper is organized as follows. In Section 2,
we point out the particular type of multidimensional quadratic BSDEs
that we study and list the assumptions we work under. In Section 3,
we firstly use properties of BMO martingales, Girsanov's theorem and
predictable representation property to show that the usual iteration
method is also well defined for our problem. Then we show that a contraction
map can be found on a suitable complete metric space on a fixed small
time interval under some constraints on parameters, which gives a
result of existence of solution to our BSDE on the whole time interval
by pasting time together. In Section 4, by using the result obtained
in Section 3 for our BSDE with small terminal values and pasting space
together, we derive a result of existence and uniqueness of solution
to the one dimensional case of our BSDE with bounded terminal values.
Finally, in Section 5, by using the result obtained in Section 4 for
the one dimensional case, we show the existence of solution to a lower
triangular quadratic BSDE with some bounded terminal values satisfying
a measurability condition.

\section{Definitions and assumptions}

Let us begin with a few notations and definitions which are nevertheless
standard in BSDE literature as follows:
\begin{itemize}
\item $\left\Vert y\right\Vert =\sqrt{y^{T}y}$ for $y\in\mbox{\ensuremath{\boldsymbol{R}}}^{d}$
and $\left\Vert z\right\Vert =\sqrt{\mbox{Tr}\left(zz^{T}\right)}$
for $z\in\mbox{\ensuremath{\boldsymbol{R}}}^{d\times k}$ denote the
Euclidean norms.
\item $\boldsymbol{E}^{\mathcal{F}_{t}}\left[\cdot\right]:=\boldsymbol{E}\left[\cdot\left|\mathcal{F}_{t}\right.\right]$
\item $W$ is a standard $k$-dimensional Brownian motion defined on $\left(\Omega,\mathcal{F},\left(\mathcal{F}_{t}\right),\mathbb{P}\right)$
and $\left(\mathcal{F}_{t}\right)_{t\in\left[0,T\right]}$ is its
Brownian filtration.
\item $\mathcal{M}^{2}\left(\mbox{\ensuremath{\boldsymbol{R}}}^{d}\right)$
and $\mathcal{M}^{2}\left(\mbox{\ensuremath{\boldsymbol{R}}}^{d\times k}\right)$
denote respectively the Banach spaces of progressively measurable
processes $Y$ and $Z$ such that $\left\Vert Y\right\Vert _{\mathcal{M}^{2}}^{2}=\boldsymbol{E}\int_{0}^{T}$$\left\Vert Y_{t}\right\Vert ^{2}$$dt<\infty$
and $\left\Vert Z\right\Vert _{\mathcal{M}^{2}}^{2}=\boldsymbol{E}\int_{0}^{T}$$\left\Vert Z_{t}\right\Vert ^{2}$$dt<\infty$. 
\item $\mathcal{S}^{\infty}\left(\mbox{\ensuremath{\boldsymbol{R}}}^{d}\right)$
denotes the Banach space of bounded progressively measurable processes
$Y$.
\item $\mathcal{S}_{C_{1}}^{\infty}\left(\mbox{\ensuremath{\boldsymbol{R}}}^{d}\right)$
denotes the collection of bounded progressively measurable processes
$Y$ such that $\left\Vert Y\right\Vert _{\mathcal{S}^{\infty}}\leq C_{1}$
and $C_{1}$ is a non negative constant.
\item A continuous square integrable martingale $M$ with $M_{0}=0$ is
a BMO martingale if 
\[
\left\Vert M\right\Vert _{BMO}^{2}=\underset{\tau}{\sup}\frac{\boldsymbol{E}\left\Vert M_{T}-M_{\tau}\right\Vert ^{2}}{\mathbb{P}\left(\tau<T\right)}<\infty
\]
where $T$ is the terminal time and the supremum is taken over all
stopping times $\tau$ bounded by $T$, with the convention that if
$\mathbb{P}\left(\tau=T\right)=1$ then $\frac{\boldsymbol{E}\left\Vert M_{T}-M_{\tau}\right\Vert ^{2}}{\mathbb{P}\left(\tau<T\right)}=0$.
\item $\mathcal{B}\left(\mbox{\ensuremath{\boldsymbol{R}}}^{d\times k}\right)$
denotes the space of progressively measurable processes $Z$ such
that $\int Z_{s}dW_{s}$ is a BMO martingale and define $\left\Vert Z\right\Vert _{\mathcal{B}}=\left\Vert \int Z_{s}dW_{s}\right\Vert _{BMO}$.
Then $\left(\mathcal{B},\left\Vert \cdot\right\Vert _{\mathcal{B}}\right)$
is a Banach space due to the fact that the space of BMO martingales
null at zero is Banach and the definition of stochastic integral.
\item $\mathcal{B}_{\sqrt{R}}\left(\mbox{\ensuremath{\boldsymbol{R}}}^{d\times k}\right)=\left\{ Z\in\mathcal{B}\left(\mbox{\ensuremath{\boldsymbol{R}}}^{d\times k}\right):\left\Vert Z\right\Vert _{\mathcal{B}}\leq\sqrt{R}\right\} $
with $R$ being a positive constant, which is a closed subset of $\mathcal{B}\left(\mbox{\ensuremath{\boldsymbol{R}}}^{d\times k}\right)$.
\item Since the definition of BMO space depends on the underlying probability
measure, we denote by BMO$\left(\mathbb{P}\right)$ the BMO space
under $\mathbb{P}$ and by BMO$\left(\mathbb{Q}\right)$ the BMO space
under $\mathbb{Q}$ respectively in case of necessity. For the same
reason, we also denote by $\mathcal{B}\left(\mbox{\ensuremath{\boldsymbol{R}}}^{d\times k}\right)\left(\mathbb{\mathbb{P}}\right)$
the space of progressively measurable processes $Z$ such that $\int Z_{s}dW_{s}\in$BMO$\left(\mathbb{P}\right)$
and by $\mathcal{B}\left(\mbox{\ensuremath{\boldsymbol{R}}}^{d\times k}\right)\left(\mathbb{Q}\right)$
the space of progressively measurable processes $Z$ such that $\int Z_{s}dW_{s}^{\mathbb{Q}}\in$BMO$\left(\mathbb{Q}\right)$
where $W^{\mathbb{Q}}$ is a standard $k$-dimensional Brownian motion
under $\mathbb{Q}$.
\end{itemize}
We consider the following BSDE:
\begin{align}
\begin{cases}
dY_{t} & =Z_{t}f\left(Y_{t},Z_{t}\right)dt+Z_{t}dW_{t}\\
Y_{T} & =\xi
\end{cases}\label{eq:4.1}
\end{align}
where $Y_{t}$ is $\mbox{\ensuremath{\boldsymbol{R}}}^{d}$-valued
, $Z_{t}$ is $\mbox{\ensuremath{\boldsymbol{R}}}^{d\times k}$-valued,
$f$ is $\mbox{\ensuremath{\boldsymbol{R}}}^{k}$-valued and $\xi$
is $\mbox{\ensuremath{\boldsymbol{R}}}^{d}$-valued and $\mathcal{F}_{T}$
measurable, which should be interpreted as a stochastic integral equation
(\ref{eq:1}).

We make the following assumptions:

$\left(\mathcal{A}1\right)$ $\left\Vert \xi\right\Vert \leq C_{1}$
for some constant $C_{1}>0$, i.e. $\xi$ is bounded.

$\left(\mathcal{A}2\right)$ $f$ satisfies the Lipschitz condition
and has a linear growth: 
\begin{align*}
\left\Vert f\left(y_{1},z_{1}\right)-f\left(y_{2},z_{2}\right)\right\Vert  & \leq C_{2}\left\Vert y_{1}-y_{2}\right\Vert +C_{3}\left\Vert z_{1}-z_{2}\right\Vert ,\\
\left\Vert f\left(y_{1},z_{1}\right)\right\Vert  & \leq C_{2}\left\Vert y_{1}\right\Vert +C_{3}\left\Vert z_{1}\right\Vert +C_{4},
\end{align*}
for any $y_{1},y_{2}\in\mbox{\ensuremath{\boldsymbol{R}}}^{d}$ and
$z_{1},z_{2}\in\mbox{\ensuremath{\boldsymbol{R}}}^{d\times k}$, so
that $zf\left(y,z\right)$ has quadratic growth in $z$. $C_{2},C_{3},C_{4}$
are non negative constants. $f\left(Y,Z\right)$ is progressively
measurable when $\left(Y,Z\right)$ is progressively measurable.

By a solution to (\ref{eq:4.1}), we mean a pair of stochastic processes
$\left(Y,Z\right)$ on $\left(\Omega,\mathcal{F},\left(\mathcal{F}_{t}\right),\mathbb{P}\right)$,
where $Y=\left(Y_{t}\right)\in\mathcal{S}^{\infty}\left(\mbox{\ensuremath{\boldsymbol{R}}}^{d}\right)$
and $Z=\left(Z_{t}\right)\in\mathcal{B}\left(\mbox{\ensuremath{\boldsymbol{R}}}^{d\times k}\right)$.
Moreover $Z\in\mathcal{B}\left(\mbox{\ensuremath{\boldsymbol{R}}}^{d\times k}\right)$
implies that $Z\in\mathcal{M}^{2}\left(\mbox{\ensuremath{\boldsymbol{R}}}^{d\times k}\right)$
due to the fact that 

\begin{equation}
\left\Vert Z\right\Vert _{\mathcal{M}^{2}}^{2}=\boldsymbol{E}\left(\int_{0}^{T}\left\Vert Z_{s}\right\Vert ^{2}ds\right)=\boldsymbol{E}\left\Vert \int_{0}^{T}Z_{s}dW_{s}\right\Vert ^{2}\leq\left\Vert \int Z_{s}dW_{s}\right\Vert _{BMO}^{2}=\left\Vert Z\right\Vert _{\mathcal{B}}^{2}.
\end{equation}

The following properties about BMO martingales are well known. If
$\int Z_{s}dW_{s}$ is a BMO martingale, then
\[
\boldsymbol{E}^{\mathcal{F}_{t}}\left(\int_{t}^{T}\left\Vert Z_{s}\right\Vert ^{2}ds\right)\leq\left\Vert \int Z_{s}dW_{s}\right\Vert _{BMO}^{2}=\left\Vert Z\right\Vert _{\mathcal{B}}^{2},\forall t\in\left[0,T\right]
\]
and if $\boldsymbol{E}^{\mathcal{F}_{t}}\left(\int_{t}^{T}\left\Vert \tilde{Z}_{s}\right\Vert ^{2}ds\right)\leq N$,
for every $t\in\left[0,T\right]$ where $N$ is a non negative constant,
then $\int\tilde{Z}_{s}dW_{s}$ is a BMO martingale and
\[
\left\Vert \tilde{Z}\right\Vert _{\mathcal{B}}^{2}=\left\Vert \int\tilde{Z}_{s}dW_{s}\right\Vert _{BMO}^{2}\leq N,
\]
see Kazamaki \cite{key-9} for details.

The following lemma is standard, whose proof can be found in Hu and
Tang \cite{key-8} for example, and it plays an important role in
some of the subsequent arguments.
\begin{lem}
\label{lem:For-K>0,-there}For K>0, there are constants $c_{1}>0$
and $c_{2}>0$ such that for any BMO martingale $M$, we have for
any BMO martingale N with $\left\Vert N\right\Vert _{BMO\left(\mathbb{P}\right)}\leq K$
that
\[
c_{1}\left\Vert M\right\Vert _{BMO\left(\mathbb{P}\right)}\leq\left\Vert \tilde{M}\right\Vert _{BMO\left(\mathbb{Q}\right)}\leq c_{2}\left\Vert M\right\Vert _{BMO\left(\mathbb{P}\right)}
\]
where $\tilde{M}=M-\left\langle M,N\right\rangle $ and $\left.\frac{d\mathbb{Q}}{d\mathbb{P}}\right|_{\mathcal{F}_{t}}={\cal E}\left(N\right)_{t}$.
\end{lem}

The following corollary can be obtained immediately by Lemma \ref{lem:For-K>0,-there}.
\begin{cor}
\label{cor:Assume-that-}Assume that $N\in$ \textup{BMO}$\left(\mathbb{P}\right)$,
then $M\in$\textup{BMO}$\left(\mathbb{P}\right)$ if and only if
$\tilde{M}\in$\textup{BMO}$\left(\mathbb{Q}\right)$, where $\tilde{M}$
and $\mathbb{Q}$ are defined as in Lemma \ref{lem:For-K>0,-there}.
\end{cor}

\section{Existence of solution}

We will use the iteration method. Let $\mathcal{S}_{C_{1}}^{\infty}\times\mathcal{B}$
denote, for simplicity, the space $\mathcal{S}_{C_{1}}^{\infty}\left(\mbox{\ensuremath{\boldsymbol{R}}}^{d}\right)\times\mathcal{B}\left(\mbox{\ensuremath{\boldsymbol{R}}}^{d\times k}\right)$.
Suppose that $\left(Y,Z\right)\in\mathcal{S}_{C_{1}}^{\infty}\times\mathcal{B}$
and $f$ and $\xi$ satisfy the above assumptions $\left(\mathcal{A}1\right)$
and $\left(\mathcal{A}2\right)$, then we have, by the linear growth
of $f$, boundedness of $Y$ and properties of the BMO martingale
$\int Z_{s}dW_{s}$, that $\int f\left(Y_{s},Z_{s}\right)^{T}dW_{s}$
is also a BMO martingale, which in turn implies that the stochastic
exponential of $-\int f\left(Y_{s},Z_{s}\right)^{T}dW_{s}$ is a martingale
on $\left[0,T\right]$. Hence we define a probability measure $\mathbb{Q}$
by 
\begin{equation}
\left.\frac{d\mathbb{Q}}{d\mathbb{P}}\right|_{\mathcal{F}_{T}}={\cal E}\left(-\int f\left(Y_{s},Z_{s}\right)^{T}dW_{s}\right)_{T}
\end{equation}
and define 
\begin{equation}
dW_{t}^{\mathbb{Q}}=dW_{t}+f\left(Y_{t},Z_{t}\right)dt.\label{eq:3.2}
\end{equation}
Then $W^{\mathbb{Q}}$ is a standard $k$-dimensional Brownian motion
under probability measure $\mathbb{Q}$. The lemma below about continuous
martingale representation is well known and its proof may be found
in Cohen and Elliott \cite{key-6}.
\begin{lem}
\label{lem:Suppose--is}Suppose $M$ is a $d$-dimensional continuous
local martingale under $\mathbb{Q}$ with $\mathbb{Q}$ defined as
above, then there exists a unique predictable process $H$ such that
$M-M_{0}=\int H_{s}dW_{s}^{\mathbb{Q}}$.
\end{lem}

\begin{proof}
Since $M$ is a continuous semi-martingale under $\mathbb{P}$, $M-M_{0}=N+A$
where $N$ is a continuous local martingale null at 0 under $\mathbb{P}$
and $A$ is a finite variation process. By the martingale representation
theorem applying to $N$, we have that $N=\int H_{s}dW_{s}$ for some
predictable process $H$. Thus $M-M_{0}-\int H_{s}dW_{s}^{\mathbb{Q}}=-\int H_{s}f\left(Y_{s},Z_{s}\right)ds+A$
by equation (\ref{eq:3.2}), which implies that the continuous $\mathbb{Q}$-local
martingale $M-M_{0}-\int H_{s}dW_{s}^{\mathbb{Q}}$ is of finite variation
and null at 0. Thus we have that $M-M_{0}=\int H_{s}dW_{s}^{\mathbb{Q}}$.
Uniqueness can be proved in the usual way.
\end{proof}
Let $\delta\in\left(0,1\right)$, and consider time interval $\left[T-\delta T,T\right]$.
Let $\left(Y,Z\right)\in\mathcal{S}_{C_{1}}^{\infty}\times\mathcal{B}$
but with duration $\left[T-\delta T,T\right]$.

Since $\left\Vert \xi\right\Vert \leq C_{1}$, $\tilde{Y_{t}}=\boldsymbol{E}_{\mathbb{Q}}^{\mathcal{F}_{t}}\left[\xi\right]$
is a continuous martingale under $\mathbb{Q}$ on $\left[T-\delta T,T\right]$.
Thus by Lemma \ref{lem:Suppose--is}, there exists a unique predictable
process $\tilde{Z}$ on $\left[T-\delta T,T\right]$ such that
\begin{align}
\begin{cases}
d\tilde{Y_{t}} & =\tilde{Z_{t}}f\left(Y_{t},Z_{t}\right)dt+\tilde{Z_{t}}dW_{t}\\
\tilde{Y_{T}} & =\xi
\end{cases}\label{eq:5.2}
\end{align}
with $\left\Vert \tilde{Y}\right\Vert _{\mathcal{S}^{\infty}}\leq C_{1}$,
which means that $\tilde{Y}\in\mathcal{S}_{C_{1}}^{\infty}\left(\mbox{\ensuremath{\boldsymbol{R}}}^{d}\right)$. 
\begin{lem}
\textup{\label{lem:BMOwhere--is}$\int\tilde{Z_{t}}dW_{t}\in$BMO}$\left(\mathbb{P}\right)$
where $\tilde{Z}$ is defined as in (\ref{eq:5.2}).
\end{lem}

\begin{proof}
Since $\tilde{Y}$ defined in (\ref{eq:5.2}) belongs to BMO$\left(\mathbb{Q}\right)$
as it is bounded under $\mathbb{Q}$, it can be derived immediately
by Corollary \ref{cor:Assume-that-} that $\int\tilde{Z_{t}}dW_{t}\in$BMO$\left(\mathbb{P}\right)$. 
\end{proof}
We prove the following proposition.
\begin{prop}
\label{thm:If--then}If $C_{1}C_{3}<e^{-\frac{1}{2}}$ where $e^{-\frac{1}{2}}$
is just a universal constant, and it does not imply that it is optimal.
Then there is a non negative constant $C_{6}$ depending on $C_{1},C_{2},C_{3},C_{4}$
and $\delta T$ such that
\[
\left\Vert \tilde{Z}\right\Vert _{\mathcal{B}}^{2}\leq C_{6}+\frac{1}{2}\left\Vert Z\right\Vert _{\mathcal{B}}^{2}
\]
for any pairs $\left(Y,Z\right)$ and $\left(\tilde{Y},\tilde{Z}\right)$
on $\left[T-\delta T,T\right]$ defined by BSDE (\ref{eq:5.2}).
\end{prop}

\begin{proof}
Consider $\varphi\left(x\right)=e^{Kx}$ where $K$ is a positive
constant to be determined later. Let $\eta_{t}=\varphi\left(\left\Vert \tilde{Y_{t}}\right\Vert ^{2}\right)$.
Then by Itô's formula we have
\begin{eqnarray*}
d\left\Vert \tilde{Y}\right\Vert ^{2} & = & 2\underset{i}{\sum}\tilde{Y^{i}}d\tilde{Y^{i}}+\left\Vert \tilde{Z}\right\Vert ^{2}dt\\
 & = & 2\underset{i,j}{\sum}f_{j}\left(Y,Z\right)\tilde{Y^{i}}\tilde{Z^{i,j}}dt+\left\Vert \tilde{Z}\right\Vert ^{2}dt+2\underset{i,j}{\sum}\tilde{Y^{i}}\tilde{Z^{i,j}}dW^{j},
\end{eqnarray*}
and 
\begin{eqnarray*}
d\eta & = & K\eta d\left\Vert \tilde{Y}\right\Vert ^{2}+2K^{2}\eta\underset{j}{\sum}\left|\underset{i}{\sum}\tilde{Y^{i}}\tilde{Z^{i,j}}\right|^{2}dt\\
 & = & K\eta\left[2\underset{i,j}{\sum}f_{j}\left(Y,Z\right)\tilde{Y^{i}}\tilde{Z^{i,j}}dt+\left\Vert \tilde{Z}\right\Vert ^{2}dt\right]+2K^{2}\eta\underset{j}{\sum}\left|\underset{i}{\sum}\tilde{Y^{i}}\tilde{Z^{i,j}}\right|^{2}dt\\
 &  & +2K\eta\underset{i,j}{\sum}\tilde{Y^{i}}\tilde{Z^{i,j}}dW^{j}.
\end{eqnarray*}
Set vector $\tilde{U}$ with $\tilde{U^{j}}=\underset{i}{\sum}\tilde{Y^{i}}\tilde{Z^{i,j}}$,
so the previous equation can be written as
\[
d\eta=K\eta\left[2\tilde{U}^{T}f\left(Y,Z\right)+\left\Vert \tilde{Z}\right\Vert ^{2}+2K\left\Vert \tilde{U}\right\Vert ^{2}\right]dt+2K\eta\tilde{U}^{T}dW.
\]
Integrating the equality above from $t$ to $T$, we obtain
\[
\eta_{t}=\eta_{T}-K\int_{t}^{T}\eta\left[2\tilde{U}^{T}f\left(Y,Z\right)+\left\Vert \tilde{Z}\right\Vert ^{2}+2K\left\Vert \tilde{U}\right\Vert ^{2}\right]ds-2K\int_{t}^{T}\eta\tilde{U}^{T}dW.
\]
Since it can be derived immediately by Lemma \ref{lem:BMOwhere--is}
and the boundedness of $\tilde{Y}$ that $\int\eta\tilde{U}^{T}dW$
is a martingale, we take the conditional expectation with respect
to $\mathcal{F}_{t}$ to get
\[
\eta_{t}=\boldsymbol{E}^{\mathcal{F}_{t}}\left(\eta_{T}\right)-K\boldsymbol{E}^{\mathcal{F}_{t}}\left(\int_{t}^{T}\eta\left[2\tilde{U}^{T}f\left(Y,Z\right)+\left\Vert \tilde{Z}\right\Vert ^{2}+2K\left\Vert \tilde{U}\right\Vert ^{2}\right]ds\right).
\]
Next applying Cauchy-Schwartz inequality, the linear growth condition
of $f$ and the bound of $Y$, we deduce that
\begin{eqnarray*}
\eta_{t} & \leq & \boldsymbol{E}^{\mathcal{F}_{t}}\left(\eta_{T}\right)-K\boldsymbol{E}^{\mathcal{F}_{t}}\left(\int_{t}^{T}\eta\left[-2\left(C_{1}C_{2}+C_{3}\left\Vert Z\right\Vert +C_{4}\right)\left\Vert \tilde{U}\right\Vert +\left\Vert \tilde{Z}\right\Vert ^{2}+2K\left\Vert \tilde{U}\right\Vert ^{2}\right]ds\right)\\
 & = & \boldsymbol{E}^{\mathcal{F}_{t}}\left(\eta_{T}\right)+2\left(C_{1}C_{2}+C_{4}\right)K\boldsymbol{E}^{\mathcal{F}_{t}}\left(\int_{t}^{T}\eta\left\Vert \tilde{U}\right\Vert ds\right)\\
 &  & +2KC_{3}\boldsymbol{E}^{\mathcal{F}_{t}}\left(\int_{t}^{T}\eta\left\Vert Z\right\Vert \left\Vert \tilde{U}\right\Vert ds\right)-K\boldsymbol{E}^{\mathcal{F}_{t}}\left(\int_{t}^{T}\eta\left[\left\Vert \tilde{Z}\right\Vert ^{2}+2K\left\Vert \tilde{U}\right\Vert ^{2}\right]ds\right).
\end{eqnarray*}
Then by applying the inequalities with $\alpha,\beta>0$ 
\[
2\left\Vert \tilde{U}\right\Vert \leq\alpha+\frac{1}{\alpha}\left\Vert \tilde{U}\right\Vert ^{2}
\]
and
\[
2\left\Vert Z\right\Vert \left\Vert \tilde{U}\right\Vert \leq\beta\left\Vert Z\right\Vert ^{2}+\frac{1}{\beta}\left\Vert \tilde{U}\right\Vert ^{2},
\]
we get that
\begin{eqnarray*}
\eta_{t} & \leq & \boldsymbol{E}^{\mathcal{F}_{t}}\left(\eta_{T}\right)+\left(C_{1}C_{2}+C_{4}\right)K\boldsymbol{E}^{\mathcal{F}_{t}}\left(\int_{t}^{T}\alpha\eta ds\right)\\
 &  & -K\boldsymbol{E}^{\mathcal{F}_{t}}\left(\int_{t}^{T}\eta\left\Vert \tilde{Z}\right\Vert ^{2}ds\right)+KC_{3}\boldsymbol{E}^{\mathcal{F}_{t}}\left(\int_{t}^{T}\eta\beta\left\Vert Z\right\Vert ^{2}ds\right)\\
 &  & -K\boldsymbol{E}^{\mathcal{F}_{t}}\left(\int_{t}^{T}\eta\left[2K-\frac{C_{3}}{\beta}-\frac{\left(C_{1}C_{2}+C_{4}\right)}{\alpha}\right]\left\Vert \tilde{U}\right\Vert ^{2}ds\right).
\end{eqnarray*}
It follows that
\begin{eqnarray}
\boldsymbol{E}^{\mathcal{F}_{t}}\left(\int_{t}^{T}\eta\left\Vert \tilde{Z}\right\Vert ^{2}ds\right) & \leq & \frac{1}{K}\boldsymbol{E}^{\mathcal{F}_{t}}\left(\eta_{T}-\eta_{t}\right)+\left(C_{1}C_{2}+C_{4}\right)\boldsymbol{E}^{\mathcal{F}_{t}}\left(\int_{t}^{T}\alpha\eta ds\right)\nonumber \\
 &  & +C_{3}\beta\boldsymbol{E}^{\mathcal{F}_{t}}\left(\int_{t}^{T}\eta\left\Vert Z\right\Vert ^{2}ds\right)\nonumber \\
 &  & -\boldsymbol{E}^{\mathcal{F}_{t}}\left(\int_{t}^{T}\eta\left[2K-\frac{C_{3}}{\beta}-\frac{\left(C_{1}C_{2}+C_{4}\right)}{\alpha}\right]\left\Vert \tilde{U}\right\Vert ^{2}ds\right).\label{eq:3.4}
\end{eqnarray}
Since $\left\Vert \tilde{Y}\right\Vert _{\mathcal{S}^{\infty}}\leq C_{1}$
we deduce that $1\leq\eta\leq e^{KC_{1}^{2}}$ . We may choose constants
such that 
\begin{equation}
2K-\frac{C_{3}}{\beta}-\frac{\left(C_{1}C_{2}+C_{4}\right)}{\alpha}\geq0\label{eq:5.3}
\end{equation}
and
\begin{equation}
C_{3}\beta e^{KC_{1}^{2}}=\frac{1}{2}.\label{eq:3.6}
\end{equation}
In order to do this, it requires that
\[
K-C_{3}^{2}e^{KC_{1}^{2}}>0
\]
which means that
\[
Ke^{-1}-C_{3}^{2}e^{KC_{1}^{2}-1}>0
\]
which is possible only if $C_{1}C_{3}<e^{-\frac{1}{2}}$. When $C_{1}C_{3}<e^{-\frac{1}{2}}$
by considering 
\[
KC_{1}^{2}-C_{1}^{2}C_{3}^{2}e^{KC_{1}^{2}}>0
\]
we can choose 
\[
K=-\frac{2}{C_{1}^{2}}\ln\left(C_{1}C_{3}\right)
\]
which implies that 
\[
K-C_{3}^{2}e^{KC_{1}^{2}}>0.
\]
Then we can deduce from the previous inequality (\ref{eq:3.4}) that
\begin{eqnarray*}
\boldsymbol{E}^{\mathcal{F}_{t}}\left(\int_{t}^{T}\left\Vert \tilde{Z}_{s}\right\Vert ^{2}ds\right) & \leq & \frac{1}{K}e^{KC_{1}^{2}}+\alpha e^{KC_{1}^{2}}\left(C_{1}C_{2}+C_{4}\right)\delta T\\
 &  & +C_{3}\beta e^{KC_{1}^{2}}\boldsymbol{E}^{\mathcal{F}_{t}}\left(\int_{t}^{T}\left\Vert Z_{s}\right\Vert ^{2}ds\right)
\end{eqnarray*}
for all $t\in[T-\delta T,T]$. Using the properties of BMO martingales
we may deduce that
\begin{eqnarray*}
\left\Vert \tilde{Z}\right\Vert _{\mathcal{B}}^{2} & \leq & C_{6}+\frac{1}{2}\left\Vert Z\right\Vert _{\mathcal{B}}^{2}
\end{eqnarray*}
where
\begin{eqnarray}
C_{6} & = & e^{KC_{1}^{2}}\left[\frac{1}{K}+\alpha\left(C_{1}C_{2}+C_{4}\right)\delta T\right]\nonumber \\
 & = & \frac{1}{C_{1}^{2}C_{3}^{2}}\left[\frac{C_{1}^{2}}{-2\ln\left(C_{1}C_{3}\right)}+\alpha\left(C_{1}C_{2}+C_{4}\right)\delta T\right].\label{eq:5.4}
\end{eqnarray}
\end{proof}
The above proposition implies that we get a pair $\left(\tilde{Y},\tilde{Z}\right)\in\mathcal{S}_{C_{1}}^{\infty}\times\mathcal{B}$
on $[T-\delta T,T]$, when $C_{1}C_{3}<e^{-\frac{1}{2}}$. In this
case we define the pair $\left(\tilde{Y},\tilde{Z}\right)=\Phi\left(Y,Z\right)$
on $[T-\delta T,T]$ and $\Phi:\mathcal{S}_{C_{1}}^{\infty}\times\mathcal{B}\rightarrow\mathcal{S}_{C_{1}}^{\infty}\times\mathcal{B}$
is well defined.

In order to get a result about the global existence of solution, we
firstly consider the time interval $[T-\delta T,T]$ for some $\delta\in\left(0,1\right)$
and try to find a contraction map on a closed subspace of $\mathcal{S}_{C_{1}}^{\infty}\times\mathcal{B}$.
This approach is inspired by the method used in Tevzadze \cite{key-12}.
Then by working backwards with respect to time intervals of length
$\delta T$, we can get our result by pasting time together.
\begin{thm}
\label{thm:There-exits-}Under the above assumptions \textup{$\left(\mathcal{A}1\right)$}
and \textup{$\left(\mathcal{A}2\right)$} on $f$ and $\xi$. If\textup{
$C_{1}C_{3}<e^{-144}$}, then for any terminal time T, there exists
a positive constant $\tilde{R}=\lceil\frac{1}{\delta}\rceil R$ with
\textup{$R=2C_{6}$} and some fixed constant $\delta\in\left(0,1\right)$
such that the BSDE\textup{ }
\begin{align}
\begin{cases}
dY_{t} & =Z_{t}f\left(Y_{t},Z_{t}\right)dt+Z_{t}dW_{t}\\
Y_{T} & =\xi
\end{cases}\label{eq:4.1-1}
\end{align}
has a solution pair\textup{ $\left(Y,Z\right)\in\mathcal{S}_{C_{1}}^{\infty}\left(\mbox{\ensuremath{\boldsymbol{R}}}^{d}\right)\times\mathcal{B}_{\sqrt{\tilde{R}}}\left(\mbox{\ensuremath{\boldsymbol{R}}}^{d\times k}\right)$}
on $\left[0,T\right]$.
\end{thm}

\begin{proof}
Let $\lambda\in\left(0,1\right)$ be a positive constant to be determined
later. We firstly consider the time interval $\left[T-\delta T,T\right]$
as above and assume that

\begin{equation}
C_{1}C_{3}<e^{-\frac{4}{\lambda^{2}}}\label{eq:3.8}
\end{equation}
which implies that $C_{1}C_{3}<e^{-\frac{1}{2}}$ and we set constant
$R$ to be 
\begin{equation}
R=2C_{6}.\label{eq:3.9}
\end{equation}
We may have the following by choosing $\delta$ small enough.
\begin{equation}
2\left(C_{1}C_{2}+C_{4}\right)\sqrt{\delta T}\vee2C_{2}\sqrt{\delta T}\sqrt{R}\vee2C_{3}\sqrt{R}\leq\lambda.\label{eq:5.9}
\end{equation}
We can do this because we have condition (\ref{eq:3.8}) and in equation
(\ref{eq:5.4}):
\[
C_{6}=\frac{1}{C_{1}^{2}C_{3}^{2}}\left[\frac{C_{1}^{2}}{-2\ln\left(C_{1}C_{3}\right)}+\alpha\left(C_{1}C_{2}+C_{4}\right)\delta T\right]
\]
which implies that
\[
C_{3}^{2}C_{6}=\frac{1}{C_{1}^{2}}\left[\frac{C_{1}^{2}}{-2\ln\left(C_{1}C_{3}\right)}+\alpha\left(C_{1}C_{2}+C_{4}\right)\delta T\right],
\]
where $\alpha$ is determined by inequality (\ref{eq:5.3}): 
\[
2K-\frac{C_{3}}{\beta}-\frac{\left(C_{1}C_{2}+C_{4}\right)}{\alpha}\geq0
\]
which can be achieved when $C_{1}C_{3}<e^{-\frac{1}{2}}$. 

Let $\mathcal{S}_{C_{1}}^{\infty}\times\mathcal{B}_{\sqrt{R}}$ denote
the space $\mathcal{S}_{C_{1}}^{\infty}\left(\mbox{\ensuremath{\boldsymbol{R}}}^{d}\right)\times\mathcal{B}_{\sqrt{R}}\left(\mbox{\ensuremath{\boldsymbol{R}}}^{d\times k}\right)$.
Since $C_{1}C_{3}<e^{-\frac{1}{2}}$ which is due to condition (\ref{eq:3.8}),
we have $\Phi:\mathcal{S}_{C_{1}}^{\infty}\times\mathcal{B}\rightarrow\mathcal{S}_{C_{1}}^{\infty}\times\mathcal{B}$
as defined above. Then for any pair $\left(Y,Z\right)\in\mathcal{S}_{C_{1}}^{\infty}\times\mathcal{B}_{\sqrt{R}}$
we can get $\left(\tilde{Y},\tilde{Z}\right)=\Phi\left(Y,Z\right)$
with $\left(\tilde{Y},\tilde{Z}\right)\in\mathcal{S}_{C_{1}}^{\infty}\times\mathcal{B}$.
By Proposition \ref{thm:If--then} we have that 
\[
\left\Vert \tilde{Z}\right\Vert _{\mathcal{B}}^{2}\leq C_{6}+\frac{1}{2}\left\Vert Z\right\Vert _{\mathcal{B}}^{2}.
\]
Together with condition (\ref{eq:3.9}) we get that 
\[
\left\Vert \tilde{Z}\right\Vert _{\mathcal{B}}^{2}\leq\frac{R}{2}+\frac{1}{2}\left\Vert Z\right\Vert _{\mathcal{B}}^{2}\leq R,
\]
which implies that $\left(\tilde{Y},\tilde{Z}\right)\in\mathcal{S}_{C_{1}}^{\infty}\times\mathcal{B}_{\sqrt{R}}$.
So that $\Phi:\mathcal{S}_{C_{1}}^{\infty}\times\mathcal{B}_{\sqrt{R}}\rightarrow\mathcal{S}_{C_{1}}^{\infty}\times\mathcal{B}_{\sqrt{R}}$
is well defined.

For any $\left(Y^{1},Z^{1}\right),\left(Y^{2},Z^{2}\right)\in\mathcal{S}_{C_{1}}^{\infty}\times\mathcal{B}_{\sqrt{R}}$,
we set $\left(\tilde{Y^{1}},\tilde{Z^{1}}\right)=\Phi\left(Y^{1},Z^{1}\right)$
and $\left(\tilde{Y^{2}},\tilde{Z^{2}}\right)=\Phi\left(Y^{2},Z^{2}\right)$.
So we have $\left(\tilde{Y^{1}},\tilde{Z^{1}}\right),\left(\tilde{Y^{2}},\tilde{Z^{2}}\right)\in\mathcal{S}_{C_{1}}^{\infty}\times\mathcal{B}_{\sqrt{R}}$.
Then by setting 
\[
\triangle=Y^{1}-Y^{2},\;\tilde{\triangle}=\tilde{Y^{1}}-\tilde{Y^{2}},\;\Lambda=Z^{1}-Z^{2},\;\tilde{\Lambda}=\tilde{Z^{1}}-\tilde{Z^{2}},
\]
we get that $\left(\triangle,\Lambda\right),\left(\tilde{\triangle},\tilde{\Lambda}\right)\in\mathcal{S}^{\infty}\left(\mbox{\ensuremath{\boldsymbol{R}}}^{d}\right)\times\mathcal{B}\left(\mbox{\ensuremath{\boldsymbol{R}}}^{d\times k}\right)$
with $\tilde{\triangle}_{T}=$0 and we also get 
\begin{align*}
d\tilde{\triangle}^{i} & =\underset{j}{\sum}\tilde{\Lambda}^{i,j}dW^{j}+\underset{j}{\sum}f_{j}\left(Y^{1},Z^{1}\right)\tilde{\Lambda}^{i,j}dt\\
 & +\underset{j}{\sum}\left[f_{j}\left(Y^{1},Z^{1}\right)-f_{j}\left(Y^{2},Z^{2}\right)\right]\tilde{Z^{2}}^{i,j}dt.
\end{align*}
Then by Itô's formula we have
\begin{align}
d\left\Vert \tilde{\triangle}\right\Vert ^{2} & =2\vartheta^{T}f\left(Y^{1},Z^{1}\right)dt+2\rho^{T}\left[f\left(Y^{1},Z^{1}\right)-f\left(Y^{2},Z^{2}\right)\right]dt\nonumber \\
 & +\left\Vert \tilde{\Lambda}\right\Vert ^{2}dt+2\vartheta^{T}dW,\label{eq:2.3.15}
\end{align}
where the components of vectors $\vartheta$ and $\rho$ are defined
as 
\begin{eqnarray*}
\vartheta^{j} & = & \underset{i}{\sum}\tilde{\triangle}^{i}\tilde{\Lambda}^{i,j},\quad\rho^{j}=\underset{i}{\sum}\tilde{\triangle}^{i}\tilde{Z^{2}}^{i,j},
\end{eqnarray*}
and $\int\vartheta^{T}dW$ is a martingale by the boundedness of $\tilde{\triangle}$
and the fact that $\tilde{\Lambda}\in\mathcal{B}\left(\mbox{\ensuremath{\boldsymbol{R}}}^{d\times k}\right)$.
Then by taking conditional expectation we get
\begin{align*}
\left\Vert \tilde{\triangle}_{t}\right\Vert ^{2}+\boldsymbol{E}^{\mathcal{F}_{t}}\left[\int_{t}^{T}\left\Vert \tilde{\Lambda}\right\Vert ^{2}ds\right] & =-2\boldsymbol{E}^{\mathcal{F}_{t}}\left[\int_{t}^{T}\vartheta^{T}f\left(Y^{1},Z^{1}\right)ds\right]\\
 & -2\boldsymbol{E}^{\mathcal{F}_{t}}\left[\int_{t}^{T}\rho^{T}\left[f\left(Y^{1},Z^{1}\right)-f\left(Y^{2},Z^{2}\right)\right]ds\right],
\end{align*}
which implies that 
\begin{align*}
\left\Vert \tilde{\triangle}_{t}\right\Vert ^{2}+\boldsymbol{E}^{\mathcal{F}_{t}}\left[\int_{t}^{T}\left\Vert \tilde{\Lambda}\right\Vert ^{2}ds\right] & \leq2\boldsymbol{E}^{\mathcal{F}_{t}}\left[\int_{t}^{T}\left\Vert \vartheta\right\Vert \left\Vert f\left(Y^{1},Z^{1}\right)\right\Vert ds\right]\\
 & +2\boldsymbol{E}^{\mathcal{F}_{t}}\left[\int_{t}^{T}\left\Vert \rho\right\Vert \left\Vert f\left(Y^{1},Z^{1}\right)-f\left(Y^{2},Z^{2}\right)\right\Vert ds\right].
\end{align*}
Together with the definition of $\vartheta$ and $\rho$, we obtain
from the inequality above that 
\begin{align*}
\left\Vert \tilde{\triangle}_{t}\right\Vert ^{2}+\boldsymbol{E}^{\mathcal{F}_{t}}\left[\int_{t}^{T}\left\Vert \tilde{\Lambda}\right\Vert ^{2}ds\right] & \leq2\boldsymbol{E}^{\mathcal{F}_{t}}\left[\int_{t}^{T}\left\Vert \tilde{\triangle}\right\Vert \left\Vert \tilde{\Lambda}\right\Vert \left\Vert f\left(Y^{1},Z^{1}\right)\right\Vert ds\right]\\
 & +2\boldsymbol{E}^{\mathcal{F}_{t}}\left[\int_{t}^{T}\left\Vert \tilde{\triangle}\right\Vert \left\Vert \tilde{Z^{2}}\right\Vert \left\Vert f\left(Y^{1},Z^{1}\right)-f\left(Y^{2},Z^{2}\right)\right\Vert ds\right]
\end{align*}
for all $t\in\left[T-\delta T,T\right]$. 

Now by using the assumptions on $f$, we conclude that 
\begin{align*}
\left\Vert \tilde{\triangle}_{t}\right\Vert ^{2}+\boldsymbol{E}^{\mathcal{F}_{t}}\left[\int_{t}^{T}\left\Vert \tilde{\Lambda}_{s}\right\Vert ^{2}ds\right]\leq & 2\boldsymbol{E}^{\mathcal{F}_{t}}\left[\int_{t}^{T}\left[\left(C_{1}C_{2}+C_{4}\right)+C_{3}\left\Vert Z_{s}^{1}\right\Vert \right]\left\Vert \tilde{\triangle}_{s}\right\Vert \left\Vert \tilde{\Lambda}_{s}\right\Vert ds\right]\\
 & +2\boldsymbol{E}^{\mathcal{F}_{t}}\left[\int_{t}^{T}\left(C_{2}\left\Vert \triangle_{s}\right\Vert +C_{3}\left\Vert \Lambda_{s}\right\Vert \right)\left\Vert \tilde{\triangle}_{s}\right\Vert \left\Vert \tilde{Z_{s}^{2}}\right\Vert ds\right]\\
\leq & 2\left(C_{1}C_{2}+C_{4}\right)\sqrt{\delta T}\left\Vert \tilde{\triangle}\right\Vert _{\mathcal{S}^{\infty}}\left\Vert \tilde{\Lambda}\right\Vert _{\mathcal{B}}+2C_{3}\left\Vert Z^{1}\right\Vert _{\mathcal{B}}\left\Vert \tilde{\triangle}\right\Vert _{\mathcal{S}^{\infty}}\left\Vert \tilde{\Lambda}\right\Vert _{\mathcal{B}}\\
 & +2C_{2}\sqrt{\delta T}\left\Vert \tilde{Z^{2}}\right\Vert _{\mathcal{B}}\left\Vert \tilde{\triangle}\right\Vert _{\mathcal{S}^{\infty}}\left\Vert \triangle\right\Vert _{\mathcal{S}^{\infty}}+2C_{3}\left\Vert \tilde{Z^{2}}\right\Vert _{\mathcal{B}}\left\Vert \tilde{\triangle}\right\Vert _{\mathcal{S}^{\infty}}\left\Vert \Lambda\right\Vert _{\mathcal{B}},
\end{align*}
from which we deduce that 
\begin{align}
\left\Vert \tilde{\triangle}\right\Vert _{\mathcal{S}^{\infty}} & \leq2\left(C_{1}C_{2}+C_{4}\right)\sqrt{\delta T}\left\Vert \tilde{\Lambda}\right\Vert _{\mathcal{B}}+2C_{3}\left\Vert Z^{1}\right\Vert _{\mathcal{B}}\left\Vert \tilde{\Lambda}\right\Vert _{\mathcal{B}}\label{eq:5.6}\\
 & +2C_{2}\sqrt{\delta T}\left\Vert \tilde{Z^{2}}\right\Vert _{\mathcal{B}}\left\Vert \triangle\right\Vert _{\mathcal{S}^{\infty}}+2C_{3}\left\Vert \tilde{Z^{2}}\right\Vert _{\mathcal{B}}\left\Vert \Lambda\right\Vert _{\mathcal{B}},\nonumber 
\end{align}
and 
\begin{align}
\left\Vert \tilde{\Lambda}\right\Vert _{\mathcal{B}}^{2} & \leq2\left(C_{1}C_{2}+C_{4}\right)\sqrt{\delta T}\left\Vert \tilde{\triangle}\right\Vert _{\mathcal{S}^{\infty}}\left\Vert \tilde{\Lambda}\right\Vert _{\mathcal{B}}+2C_{3}\left\Vert Z^{1}\right\Vert _{\mathcal{B}}\left\Vert \tilde{\triangle}\right\Vert _{\mathcal{S}^{\infty}}\left\Vert \tilde{\Lambda}\right\Vert _{\mathcal{B}}\label{eq:5.7}\\
 & +2C_{2}\sqrt{\delta T}\left\Vert \tilde{Z^{2}}\right\Vert _{\mathcal{B}}\left\Vert \tilde{\triangle}\right\Vert _{\mathcal{S}^{\infty}}\left\Vert \triangle\right\Vert _{\mathcal{S}^{\infty}}+2C_{3}\left\Vert \tilde{Z^{2}}\right\Vert _{\mathcal{B}}\left\Vert \tilde{\triangle}\right\Vert _{\mathcal{S}^{\infty}}\left\Vert \Lambda\right\Vert _{\mathcal{B}}.\nonumber 
\end{align}
Thus we have that 
\begin{align}
\left\Vert \tilde{\triangle}\right\Vert _{\mathcal{S}^{\infty}} & \leq2\left(C_{1}C_{2}+C_{4}\right)\sqrt{\delta T}\left\Vert \tilde{\Lambda}\right\Vert _{\mathcal{B}}+2C_{3}\sqrt{R}\left\Vert \tilde{\Lambda}\right\Vert _{\mathcal{B}}\label{eq:5.6-1}\\
 & +2C_{2}\sqrt{\delta T}\sqrt{R}\left\Vert \triangle\right\Vert _{\mathcal{S}^{\infty}}+2C_{3}\sqrt{R}\left\Vert \Lambda\right\Vert _{\mathcal{B}},\nonumber 
\end{align}
and  
\begin{align}
\left\Vert \tilde{\Lambda}\right\Vert _{\mathcal{B}}^{2} & \leq2\left(C_{1}C_{2}+C_{4}\right)\sqrt{\delta T}\left\Vert \tilde{\triangle}\right\Vert _{\mathcal{S}^{\infty}}\left\Vert \tilde{\Lambda}\right\Vert _{\mathcal{B}}+2C_{3}\sqrt{R}\left\Vert \tilde{\triangle}\right\Vert _{\mathcal{S}^{\infty}}\left\Vert \tilde{\Lambda}\right\Vert _{\mathcal{B}}\label{eq:5.7-1}\\
 & +2C_{2}\sqrt{\delta T}\sqrt{R}\left\Vert \tilde{\triangle}\right\Vert _{\mathcal{S}^{\infty}}\left\Vert \triangle\right\Vert _{\mathcal{S}^{\infty}}+2C_{3}\sqrt{R}\left\Vert \tilde{\triangle}\right\Vert _{\mathcal{S}^{\infty}}\left\Vert \Lambda\right\Vert _{\mathcal{B}}.\nonumber 
\end{align}
Then by condition (\ref{eq:5.9}) we get that 
\begin{align}
\left\Vert \tilde{\triangle}\right\Vert _{\mathcal{S}^{\infty}} & \leq\lambda\left(2\left\Vert \tilde{\Lambda}\right\Vert _{\mathcal{B}}+\left\Vert \triangle\right\Vert _{\mathcal{S}^{\infty}}+\left\Vert \Lambda\right\Vert _{\mathcal{B}}\right)\label{eq:5.6-1-1}
\end{align}
and 
\begin{align}
\left\Vert \tilde{\Lambda}\right\Vert _{\mathcal{B}}^{2} & \leq\lambda\left\Vert \tilde{\triangle}\right\Vert _{\mathcal{S}^{\infty}}\left(2\left\Vert \tilde{\Lambda}\right\Vert _{\mathcal{B}}+\left\Vert \triangle\right\Vert _{\mathcal{S}^{\infty}}+\left\Vert \Lambda\right\Vert _{\mathcal{B}}\right).\label{eq:5.7-1-1}
\end{align}
So by substituting $\left\Vert \tilde{\triangle}\right\Vert _{\mathcal{S}^{\infty}}$
in (\ref{eq:5.7-1-1}) with (\ref{eq:5.6-1-1}) we have that
\begin{align*}
\left\Vert \tilde{\Lambda}\right\Vert _{\mathcal{B}} & \leq\lambda\left(2\left\Vert \tilde{\Lambda}\right\Vert _{\mathcal{B}}+\left\Vert \triangle\right\Vert _{\mathcal{S}^{\infty}}+\left\Vert \Lambda\right\Vert _{\mathcal{B}}\right).
\end{align*}
By combining with (\ref{eq:5.6-1-1}) we deduce that 
\begin{align*}
\left\Vert \tilde{\triangle}\right\Vert _{\mathcal{S}^{\infty}}+\left\Vert \tilde{\Lambda}\right\Vert _{\mathcal{B}} & \leq4\lambda\left\Vert \tilde{\Lambda}\right\Vert _{\mathcal{B}}+2\lambda\left(\left\Vert \triangle\right\Vert _{\mathcal{S}^{\infty}}+\left\Vert \Lambda\right\Vert _{\mathcal{B}}\right).
\end{align*}
If $\lambda<\frac{1}{4}$ we have that 
\begin{align}
\left\Vert \tilde{\triangle}\right\Vert _{\mathcal{S}^{\infty}}+\left\Vert \tilde{\Lambda}\right\Vert _{\mathcal{B}} & \leq\frac{2\lambda}{\left(1-4\lambda\right)}\left(\left\Vert \triangle\right\Vert _{\mathcal{S}^{\infty}}+\left\Vert \Lambda\right\Vert _{\mathcal{B}}\right).\label{eq:2.3.21}
\end{align}
If we can choose $\lambda<\frac{1}{6}$ so that $\frac{2\lambda}{\left(1-4\lambda\right)}<1$.
Then $\Phi:\mathcal{S}_{C_{1}}^{\infty}\times\mathcal{B}_{\sqrt{R}}\rightarrow\mathcal{S}_{C_{1}}^{\infty}\times\mathcal{B}_{\sqrt{R}}$
is a contraction map. By the Banach's fixed point theorem, we deduce
that there exists a unique solution pair $\left(Y,Z\right)\in\mathcal{S}_{C_{1}}^{\infty}\times\mathcal{B}_{\sqrt{R}}$
on the time interval $\left[T-\delta T,T\right]$ to the BSDE (\ref{eq:4.1-1})
restricted on $\mathcal{S}_{C_{1}}^{\infty}\times\mathcal{B}_{\sqrt{R}}$.
Therefore what left to be shown is that there exists $\lambda<\frac{1}{6}$
such that assumption (\ref{eq:3.8}) holds, and this can be achieved
when $C_{1}C_{3}<e^{-\frac{4}{\left(\frac{1}{6}\right)^{2}}}=e^{-144}$.

We then consider the time interval $\left[T-2\delta T,T-\delta T\right]$
if $T-2\delta T>0$ and $\left[0,T-\delta T\right]$ otherwise, and
set terminal value $\xi$ at time $T-\delta T$ to be $Y_{T-\delta T}$
which is the initial value of the solution $Y$ solved above on the
time interval $\left[T-\delta T,T\right]$. Then by using the above
same method, we get a unique solution pair in $\mathcal{S}_{C_{1}}^{\infty}\times\mathcal{B}_{\sqrt{R}}$
on the time interval $\left[T-2\delta T,T-\delta T\right]$ to the
BSDE (\ref{eq:4.1-1}). By repeating this procedure backwards and
pasting the solutions on all the time intervals together we get a
solution pair $\left(Y,Z\right)\in\mathcal{S}_{C_{1}}^{\infty}\left(\mbox{\ensuremath{\boldsymbol{R}}}^{d}\right)\times\mathcal{B}_{\sqrt{\tilde{R}}}\left(\mbox{\ensuremath{\boldsymbol{R}}}^{d\times k}\right)$
with $\tilde{R}=\lceil\frac{1}{\delta}\rceil R$ on the time interval
$\left[0,T\right]$ to the BSDE (\ref{eq:4.1-1}).

If $\delta$ and $\lambda$ satisfy the condition that $\sqrt{\lceil\frac{1}{\delta}\rceil}\lambda<\frac{1}{6}$,
which may be achievable when $T$ and $C_{1}C_{3}$ are small enough.
Then the solution pair $\left(Y,Z\right)\in\mathcal{S}_{C_{1}}^{\infty}\left(\mbox{\ensuremath{\boldsymbol{R}}}^{d}\right)\times\mathcal{B}_{\sqrt{\tilde{R}}}\left(\mbox{\ensuremath{\boldsymbol{R}}}^{d\times k}\right)$
on the time interval $\left[0,T\right]$ to the BSDE (\ref{eq:4.1-1})
is unique. This uniqueness of the solution can be proved as follows.
Suppose there exist two pairs of solutions $\left(Y^{1},Z^{1}\right),\left(Y^{2},Z^{2}\right)\in\mathcal{S}_{C_{1}}^{\infty}\left(\mbox{\ensuremath{\boldsymbol{R}}}^{d}\right)\times\mathcal{B}_{\sqrt{\tilde{R}}}\left(\mbox{\ensuremath{\boldsymbol{R}}}^{d\times k}\right)$
on the time interval $\left[0,T\right]$ to the BSDE (\ref{eq:4.1-1}).
By setting
\[
\triangle=Y^{1}-Y^{2},\Lambda=Z^{1}-Z^{2},
\]
we get that $\left(\triangle,\Lambda\right)\in\mathcal{S}^{\infty}\left(\mbox{\ensuremath{\boldsymbol{R}}}^{d}\right)\times\mathcal{B}\left(\mbox{\ensuremath{\boldsymbol{R}}}^{d\times k}\right)$
with $\triangle_{T}=$0. Then on the time interval $\left[T-\delta T,T\right]$
by repeating the procedure starting from equation (\ref{eq:2.3.15}),
we obtain an inequality which is similar to inequality (\ref{eq:2.3.21})
as follows:
\[
\left\Vert \triangle\right\Vert _{\mathcal{S}^{\infty}}+\left\Vert \Lambda\right\Vert _{\mathcal{B}}\leq\frac{2\sqrt{\lceil\frac{1}{\delta}\rceil}\lambda}{\left(1-4\sqrt{\lceil\frac{1}{\delta}\rceil}\lambda\right)}\left(\left\Vert \triangle\right\Vert _{\mathcal{S}^{\infty}}+\left\Vert \Lambda\right\Vert _{\mathcal{B}}\right).
\]
Since $\lambda<\frac{1}{6\sqrt{\lceil\frac{1}{\delta}\rceil}}$ so
that $\frac{2\sqrt{\lceil\frac{1}{\delta}\rceil}\lambda}{\left(1-4\sqrt{\lceil\frac{1}{\delta}\rceil}\lambda\right)}<1$,
we deduce that $\left\Vert \triangle\right\Vert _{\mathcal{S}^{\infty}}=0$
and $\left\Vert \Lambda\right\Vert _{\mathcal{B}}=0$. Thus $\left(Y^{1},Z^{1}\right)$
equals $\left(Y^{2},Z^{2}\right)$ on the time interval $\left[T-\delta T,T\right]$,
in particular $Y_{T-\delta T}^{1}$ equals $Y_{T-\delta T}^{2}$.
We then consider the time interval $\left[T-2\delta T,T-\delta T\right]$
if $T-2\delta T>0$ and $\left[0,T-\delta T\right]$ otherwise, and
terminal values at time $T-\delta T$ are $Y_{T-\delta T}^{1}$ and
$Y_{T-\delta T}^{2}$ respectively for the two solutions. Again by
using the same procedure starting from equation (\ref{eq:2.3.15}),
we deduce that $\left(Y^{1},Z^{1}\right)$ equals $\left(Y^{2},Z^{2}\right)$
on the time interval $\left[T-2\delta T,T-\delta T\right]$ as well.
Thus by repeating this procedure backwards, we conclude that $\left(Y^{1},Z^{1}\right)$
equals $\left(Y^{2},Z^{2}\right)$ on the time interval $\left[0,T\right]$.
\end{proof}
\begin{rem}
Theorem \ref{thm:There-exits-} says that, given parameters $C_{2},C_{3},C_{4}$,
if the bound $C_{1}$ of the terminal value is small enough then there
exists a solution pair $\left(Y,Z\right)\in\mathcal{S}^{\infty}\left(\mbox{\ensuremath{\boldsymbol{R}}}^{d}\right)\times\mathcal{B}\left(\mbox{\ensuremath{\boldsymbol{R}}}^{d\times k}\right)$
on the time interval $\left[0,T\right]$ to the BSDE (\ref{eq:4.1-1}).
\end{rem}

\begin{thm}
\label{thm:Suppose--where}Suppose $\left.\frac{d\mathbb{Q}}{d\mathbb{P}}\right|_{\mathcal{F}_{t}}={\cal E}\left(N\right)_{t}$
where $N\in$ \textup{BMO}$\left(\mathbb{P}\right)$, and $f$ and
$\xi$ satisfy the same above assumptions \textup{$\left(\mathcal{A}1\right)$}
and \textup{$\left(\mathcal{A}2\right)$} with\textup{ $C_{1}C_{3}<e^{-144}$},
then the BSDE\textup{ }
\begin{align}
\begin{cases}
dY_{t} & =Z_{t}f\left(Y_{t},Z_{t}\right)dt+Z_{t}dW_{t}^{\mathbb{Q}}\\
Y_{T} & =\xi
\end{cases}\label{eq:4.1-1-1}
\end{align}
where $W^{\mathbb{Q}}$ is a standard $k$-dimensional Brownian motion
under $\mathbb{Q}$ defined as
\[
dW_{t}^{\mathbb{Q}}=dW_{t}-d\left\langle W,N\right\rangle _{t},
\]
has a solution pair\textup{ $\left(Y,Z\right)\in\mathcal{S}^{\infty}\left(\mbox{\ensuremath{\boldsymbol{R}}}^{d}\right)\times\mathcal{B}\left(\mbox{\ensuremath{\boldsymbol{R}}}^{d\times k}\right)\left(\mathbb{P}\right)$}
on $\left[0,T\right]$.
\end{thm}

\begin{proof}
It can be seen clearly that the above proof also works under probability
measure $\mathbb{Q}$ with $W^{\mathbb{Q}}$ instead of probability
measure $\mathbb{P}$ with $W$. Thus there exists a solution pair
$\left(Y,Z\right)\in\mathcal{S}^{\infty}\left(\mbox{\ensuremath{\boldsymbol{R}}}^{d}\right)\times\mathcal{B}\left(\mbox{\ensuremath{\boldsymbol{R}}}^{d\times k}\right)\left(\mathbb{Q}\right)$
on the time interval $\left[0,T\right]$ to the BSDE (\ref{eq:4.1-1-1})
by Theorem \ref{thm:There-exits-}. It means that $\int Z_{s}dW_{s}^{\mathbb{Q}}\in$BMO$\left(\mathbb{Q}\right)$,
which implies that $\int Z_{s}dW_{s}\in$BMO$\left(\mathbb{P}\right)$
by Corollary \ref{cor:Assume-that-}. Thus we deduce that $Z\in\mathcal{B}\left(\mbox{\ensuremath{\boldsymbol{R}}}^{d\times k}\right)\left(\mathbb{\mathbb{P}}\right)$
and $\left(Y,Z\right)\in\mathcal{S}^{\infty}\left(\mbox{\ensuremath{\boldsymbol{R}}}^{d}\right)\times\mathcal{B}\left(\mbox{\ensuremath{\boldsymbol{R}}}^{d\times k}\right)\left(\mathbb{P}\right)$.
\end{proof}

\section{One dimensional case with bounded terminal values}

As an application of the results in the previous section, we prove
the existence of solution for the one dimensional case of our BSDE
with bounded terminal values by pasting space together and this approach
is also used in Tevzadze \cite{key-12}.

We consider the one dimensional case i.e. $d=1$.
\begin{lem}
\label{thm:Given-,-suppose}Given \textup{$\hat{Z}\in$$\mathcal{B}\left(\mbox{\ensuremath{\boldsymbol{R}}}^{1\times k}\right)\left(\mathbb{P}\right)$},
suppose $f$ and $\xi$ satisfy the above assumptions \textup{$\left(\mathcal{A}1\right)$}
and \textup{$\left(\mathcal{A}2\right)$ with} \textup{$C_{1}C_{3}<e^{-144}$,
}then the BSDE 
\begin{equation}
\begin{cases}
dY_{t} & =\left[\left(\hat{Z_{t}}+Z_{t}\right)f\left(\hat{Z_{t}}+Z_{t}\right)-\left(\hat{Z_{t}}\right)f\left(\hat{Z_{t}}\right)\right]dt+Z_{t}dW_{t}\\
Y_{T} & =\xi
\end{cases}\label{eq:4.1-2}
\end{equation}
has a solution pair\textup{ $\left(Y,Z\right)\in\mathcal{S}^{\infty}\left(\mbox{\ensuremath{\boldsymbol{R}}}\right)\times\mathcal{B}\left(\mbox{\ensuremath{\boldsymbol{R}}}^{1\times k}\right)\left(\mathbb{P}\right)$}
on $\left[0,T\right]$.
\end{lem}

\begin{proof}
We rearrange the terms to get
\begin{align*}
dY_{t} & =Z_{t}f\left(\hat{Z_{t}}+Z_{t}\right)dt+\hat{Z_{t}}\left[f\left(\hat{Z_{t}}+Z_{t}\right)-f\left(\hat{Z_{t}}\right)\right]dt+Z_{t}dW_{t}\\
 & =Z_{t}\left[f\left(\hat{Z_{t}}+Z_{t}\right)-f\left(\hat{Z_{t}}\right)\right]dt+Z_{t}f\left(\hat{Z_{t}}\right)dt+\hat{Z_{t}}\left[f\left(\hat{Z_{t}}+Z_{t}\right)-f\left(\hat{Z_{t}}\right)\right]dt+Z_{t}dW_{t}.
\end{align*}
For all $z\in\mbox{\ensuremath{\boldsymbol{R}}}^{1\times k}$ we define
$g$ as 
\[
g\left(z\right)=f\left(\hat{Z}+z\right)-f\left(\hat{Z}\right),
\]
then it can be verified directly that $g$ satisfies the above assumption
$\left(\mathcal{A}2\right)$ with the same parameter $C_{3}$ as that
of $f$ and we have 
\begin{align*}
dY_{t} & =Z_{t}g\left(Z_{t}\right)dt+Z_{t}f\left(\hat{Z_{t}}\right)dt+\hat{Z_{t}}\left[f\left(\hat{Z_{t}}+Z_{t}\right)-f\left(\hat{Z_{t}}\right)\right]dt+Z_{t}dW_{t}.
\end{align*}
By a similar argument used in Hu and Tang \cite{key-8}, for $i=1,2\cdots k$,
we can define a vector process $\beta\left(i\right)$ taking values
in $\mbox{\ensuremath{\boldsymbol{R}}}^{k\times1}$ with $\left\Vert \beta\left(i\right)\right\Vert ^{2}\leq kC_{3}^{2}$
such that 
\[
f_{i}\left(\hat{Z}+Z\right)-f_{i}\left(\hat{Z}\right)=Z\beta\left(i\right),
\]
where $f_{i}$ is the $i$th component of $f$. Then we may define
a process $\beta$ taking values in $\mbox{\ensuremath{\boldsymbol{R}}}^{k\times k}$
where the $i$th column of $\beta$ is $\beta\left(i\right)$ and
we deduce that $\left\Vert \beta\right\Vert ^{2}\leq k^{2}C_{3}^{2}$.
It implies that
\begin{equation}
\left[f\left(\hat{Z_{t}}+Z_{t}\right)-f\left(\hat{Z_{t}}\right)\right]=\left(Z_{t}\beta_{t}\right)^{T}.\label{eq:4.2}
\end{equation}
Thus we get 
\[
dY_{t}=Z_{t}g\left(Z_{t}\right)dt+Z_{t}f\left(\hat{Z_{t}}\right)dt+\hat{Z_{t}}\left(Z_{t}\beta_{t}\right)^{T}dt+Z_{t}dW_{t},
\]
which can be written as 
\begin{align*}
dY_{t} & =Z_{t}g\left(Z_{t}\right)dt+Z_{t}f\left(\hat{Z_{t}}\right)dt+Z_{t}\beta_{t}\left(\hat{Z_{t}}\right)^{T}dt+Z_{t}dW_{t}\\
 & =Z_{t}g\left(Z_{t}\right)dt+Z_{t}\left(\left[f\left(\hat{Z_{t}}\right)+\beta_{t}\left(\hat{Z_{t}}\right)^{T}\right]dt+dW_{t}\right)\\
 & =Z_{t}g\left(Z_{t}\right)dt+Z_{t}dW_{t}^{\mathbb{Q}}
\end{align*}
where the probability measure $\mathbb{Q}$ is defined by 
\begin{equation}
\left.\frac{d\mathbb{Q}}{d\mathbb{P}}\right|_{\mathcal{F}_{T}}={\cal E}\left(-\int\left[f\left(\hat{Z_{s}}\right)^{T}+\hat{Z_{s}}\beta_{s}^{T}\right]dW_{s}\right)_{T}
\end{equation}
and it can be verified that $-\int\left[f\left(\hat{Z_{s}}\right)^{T}+\hat{Z_{s}}\beta_{s}^{T}\right]dW_{s}\in$BMO$\left(\mathbb{P}\right)$
as $\hat{Z}\in$$\mathcal{B}\left(\mbox{\ensuremath{\boldsymbol{R}}}^{1\times k}\right)\left(\mathbb{P}\right)$
and $\beta$ is bounded. $W^{\mathbb{Q}}$ is defined as
\begin{equation}
dW_{t}^{\mathbb{Q}}=dW_{t}+\left[f\left(\hat{Z_{t}}\right)+\beta_{t}\left(\hat{Z_{t}}\right)^{T}\right]dt,\label{eq:3.2-1}
\end{equation}
which is a standard $k$-dimensional Brownian motion under $\mathbb{Q}$.
Since $C_{1}C_{3}<e^{-144}$ then by Theorem \ref{thm:Suppose--where}
the BSDE (\ref{eq:4.1-2}) has a solution pair $\left(Y,Z\right)\in\mathcal{S}^{\infty}\left(\mbox{\ensuremath{\boldsymbol{R}}}\right)\times\mathcal{B}\left(\mbox{\ensuremath{\boldsymbol{R}}}^{1\times k}\right)\left(\mathbb{P}\right)$
on $\left[0,T\right]$.
\end{proof}
\begin{thm}
\label{thm:When-d=00003D1,-suppose}When d=1, suppose $f$ and $\xi$
satisfy the above assumptions \textup{$\left(\mathcal{A}1\right)$}
and \textup{$\left(\mathcal{A}2\right)$,} then the BSDE\textup{ }
\begin{align}
\begin{cases}
dY_{t} & =Z_{t}f\left(Z_{t}\right)dt+Z_{t}dW_{t}\\
Y_{T} & =\xi
\end{cases}\label{eq:4.1-1-2}
\end{align}
has a unique solution pair\textup{ $\left(Y,Z\right)\in\mathcal{S}^{\infty}\left(\mbox{\ensuremath{\boldsymbol{R}}}\right)\times\mathcal{B}\left(\mbox{\ensuremath{\boldsymbol{R}}}^{1\times k}\right)$}
on $\left[0,T\right]$.
\end{thm}

\begin{proof}
Given any $C_{1}>0$, we can find $n$ large enough such that $\frac{1}{n}C_{1}C_{3}<e^{-144}$.
By Theorem \ref{thm:There-exits-} the following BSDE
\begin{align}
\begin{cases}
dY_{t}^{1} & =Z_{t}^{1}f\left(Z_{t}^{1}\right)dt+Z_{t}^{1}dW_{t}\\
Y_{T}^{1} & =\frac{\xi}{n}
\end{cases}\label{eq:4.1-1-2-1}
\end{align}
has a solution pair $\left(Y^{1},Z^{1}\right)\in\mathcal{S}^{\infty}\left(\mbox{\ensuremath{\boldsymbol{R}}}\right)\times\mathcal{B}\left(\mbox{\ensuremath{\boldsymbol{R}}}^{1\times k}\right)$
on $\left[0,T\right]$. Then by using induction we can show that for
$m=2,\cdots,n$ the following BSDE 
\begin{equation}
\begin{cases}
dY_{t}^{m} & =\left[\left(\stackrel[j=1]{m-1}{\sum}Z_{t}^{j}+Z_{t}^{m}\right)f\left(\stackrel[j=1]{m-1}{\sum}Z_{t}^{j}+Z_{t}^{m}\right)-\left(\stackrel[j=1]{m-1}{\sum}Z_{t}^{j}\right)f\left(\stackrel[j=1]{m-1}{\sum}Z_{t}^{j}\right)\right]dt+Z_{t}^{m}dW_{t}\\
Y_{T}^{m} & =\frac{\xi}{n}
\end{cases}\label{eq:4.1-2-1}
\end{equation}
has a solution pair $\left(Y^{m},Z^{m}\right)\in\mathcal{S}^{\infty}\left(\mbox{\ensuremath{\boldsymbol{R}}}\right)\times\mathcal{B}\left(\mbox{\ensuremath{\boldsymbol{R}}}^{1\times k}\right)$
on $\left[0,T\right]$ by Lemma \ref{thm:Given-,-suppose}. By adding
$Y^{i}$ and $Z^{i}$ together, i.e. letting 
\[
Z=\stackrel[j=1]{n}{\sum}Z^{j},Y=\stackrel[j=1]{n}{\sum}Y^{j},
\]
we get that 
\begin{equation}
\begin{cases}
dY_{t} & =Z_{t}f\left(Z_{t}\right)dt+Z_{t}dW_{t}\\
Y_{T} & =\xi
\end{cases},\label{eq:4.1-2-1-1}
\end{equation}
with $\left(Y,Z\right)\in\mathcal{S}^{\infty}\left(\mbox{\ensuremath{\boldsymbol{R}}}\right)\times\mathcal{B}\left(\mbox{\ensuremath{\boldsymbol{R}}}^{1\times k}\right)$
on $\left[0,T\right]$.

The uniqueness of the solution can be proved as follows. Suppose there
exist two pairs of solutions $\left(Y,Z\right),\left(\hat{Y},\hat{Z}\right)\in\mathcal{S}^{\infty}\left(\mbox{\ensuremath{\boldsymbol{R}}}\right)\times\mathcal{B}\left(\mbox{\ensuremath{\boldsymbol{R}}}^{1\times k}\right)$
on $\left[0,T\right]$ to the BSDE (\ref{eq:4.1-1-2}). By the same
argument used in (\ref{eq:4.2}), we can define a process $\beta$
taking values in $\mbox{\ensuremath{\boldsymbol{R}}}^{k\times k}$
with $\left\Vert \beta\right\Vert ^{2}\leq k^{2}C_{3}^{2}$ such that
\[
\left[f\left(Z_{t}\right)-f\left(\hat{Z_{t}}\right)\right]=\left[\left(Z_{t}-\hat{Z_{t}}\right)\beta_{t}\right]^{T}.
\]
It can be verified that $-\int\left[Z_{s}\beta_{s}^{T}+f\left(\hat{Z_{s}}\right)^{T}\right]dW_{s}\in$BMO$\left(\mathbb{P}\right)$
as $Z$ and $\hat{Z}$ belong to $\mathcal{B}\left(\mbox{\ensuremath{\boldsymbol{R}}}^{1\times k}\right)\left(\mathbb{P}\right)$
and $\beta$ is bounded. We may define probability measure $\mathbb{Q}$
by 
\begin{equation}
\left.\frac{d\mathbb{Q}}{d\mathbb{P}}\right|_{\mathcal{F}_{T}}={\cal E}\left(-\int\left[Z_{s}\beta_{s}^{T}+f\left(\hat{Z_{s}}\right)^{T}\right]dW_{s}\right)_{T}.
\end{equation}
$W^{\mathbb{Q}}$ is defined as
\begin{equation}
dW_{t}^{\mathbb{Q}}=dW_{t}+\left[\beta_{t}\left(Z_{t}\right)^{T}+f\left(\hat{Z_{t}}\right)\right]dt,\label{eq:3.2-1-1}
\end{equation}
which is a standard $k$-dimensional Brownian motion under $\mathbb{Q}$.
Then we have that
\begin{align*}
Y_{t}-\hat{Y}_{t} & =-\int_{t}^{T}\left[Z_{s}f\left(Z_{s}\right)-\hat{Z_{s}}f\left(\hat{Z_{s}}\right)\right]ds-\int_{t}^{T}\left(Z_{s}-\hat{Z_{s}}\right)dW_{s}\\
 & =-\int_{t}^{T}Z_{s}\left[f\left(Z_{s}\right)-f\left(\hat{Z_{s}}\right)\right]ds-\int_{t}^{T}\left(Z_{s}-\hat{Z_{s}}\right)f\left(\hat{Z_{s}}\right)ds-\int_{t}^{T}\left(Z_{s}-\hat{Z_{s}}\right)dW_{s}\\
 & =-\int_{t}^{T}Z_{s}\left[\left(Z_{s}-\hat{Z_{s}}\right)\beta_{s}\right]^{T}ds-\int_{t}^{T}\left(Z_{s}-\hat{Z_{s}}\right)f\left(\hat{Z_{s}}\right)ds-\int_{t}^{T}\left(Z_{s}-\hat{Z_{s}}\right)dW_{s}\\
 & =-\int_{t}^{T}\left(Z_{s}-\hat{Z_{s}}\right)dW_{s}^{\mathbb{Q}}.
\end{align*}
Since it can be verified by Corollary \ref{cor:Assume-that-} that
$-\int\left(Z_{s}-\hat{Z_{s}}\right)dW_{s}^{\mathbb{Q}}\in$BMO$\left(\mathbb{Q}\right)$,
then by taking the conditional expectation with respect to $\mathcal{F}_{t}$
under $\mathbb{Q}$ for $t\in\left[0,T\right]$ we get that $Y$ equals
$\hat{Y}$. Thus we also have that
\[
\boldsymbol{E}_{\mathbb{Q}}^{\mathcal{F}_{t}}\left(\int_{t}^{T}\left\Vert Z_{s}-\hat{Z_{s}}\right\Vert ^{2}ds\right)=0,
\]
for every $t\in\left[0,T\right]$, which implies that $Z$ equals
$\hat{Z}$.
\end{proof}

\section{A lower triangular quadratic example with bounded terminal values}

We consider the case when $d=k$. Let $\mathcal{F}^{i}$ be the Brownian
filtration of $W^{i}$ which is the $i$th component of a standard
$k$-dimensional Brownian motion $W$. Then by considering each $i$
as a one dimensional case and working with respect to $\mathcal{F}^{i}$
for $i=1,2\cdots k$, we deduce that the BSDE
\begin{align}
\begin{cases}
d\hat{Y}_{t}^{i} & =Z_{t}^{i}f_{i}\left(Z_{t}^{i}\right)dt+Z_{t}^{i}dW_{t}^{i}\\
\hat{Y}_{T}^{i} & =\xi^{i}-\xi^{i-1}
\end{cases}\label{eq:4.1-1-2-2}
\end{align}
where $f_{i}$ and $\xi^{i}$ satisfy the above assumptions $\left(\mathcal{A}1\right)$
and $\left(\mathcal{A}2\right)$ with $\xi^{0}=0$ and $\xi^{i}-\xi^{i-1}$
is $\mathcal{F}_{T}^{i}$ measurable for $i=1,2\cdots k$, has a solution
pair $\left(\hat{Y}^{i},Z^{i}\right)\in\mathcal{S}^{\infty}\left(\mbox{\ensuremath{\boldsymbol{R}}}\right)\times\mathcal{B}\left(\mbox{\ensuremath{\boldsymbol{R}}}\right)$
on $\left[0,T\right]$ by Theorem \ref{thm:When-d=00003D1,-suppose}.
Then $Y^{i}=\stackrel[j=1]{i}{\sum}\hat{Y}^{j}$ solves the following
BSDE:
\begin{align}
\begin{cases}
dY_{t}^{i} & =\stackrel[j=1]{i}{\sum}Z_{t}^{j}f_{j}\left(Z_{t}^{j}\right)dt+\stackrel[j=1]{i}{\sum}Z_{t}^{j}dW_{t}^{j}\\
Y_{T}^{i} & =\xi^{i}
\end{cases}\label{eq:4.1-1-2-2-2}
\end{align}
for $i=1,2\cdots k$. Then let $Y^{i}$ to be the $i$th component
of $Y$, $\xi^{i}$ to be the $i$th component of $\xi$ and $f_{i}\left(z_{i,i}\right)$
to be the $i$th component of $f\left(z\right)$ for all $z\in\mbox{\ensuremath{\boldsymbol{R}}}^{k\times k}$,
we have by defining the lower triangular $Z$ as follows: 
\[
Z=\left[\begin{array}{cccc}
Z^{1}\\
Z^{1} & Z^{2}\\
\vdots &  & \ddots\\
Z^{1} & Z^{2} & \cdots & Z^{k}
\end{array}\right]
\]
that $\left(Y,Z\right)\in\mathcal{S}^{\infty}\left(\mbox{\ensuremath{\boldsymbol{R}^{k}}}\right)\times\mathcal{B}\left(\mbox{\ensuremath{\boldsymbol{R}}}^{k\times k}\right)$
on $\left[0,T\right]$ is a solution pair to the following quadratic
BSDE

\begin{align}
\begin{cases}
dY_{t} & =Z_{t}f\left(Z_{t}\right)dt+Z_{t}dW_{t}\\
Y_{T} & =\xi.
\end{cases}\label{eq:4.1-1-2-2-1}
\end{align}

\begin{rem}
$\xi$ in (\ref{eq:4.1-1-2-2-1}) can be any bounded terminal value
satisfying the condition that $\xi^{i}-\xi^{i-1}$ is $\mathcal{F}_{T}^{i}$
measurable where $\xi^{i}$ is the $i$th component of $\xi$ with
$\xi^{0}=0$ for $i=1,2\cdots k$.
\end{rem}

\end{document}